\documentclass[12pt]{article}
\usepackage{graphicx}
\usepackage{xcolor}
\usepackage{amssymb,amsfonts,amsthm,amsmath,calligra,tikz}
\usepackage[utf8]{inputenc}
\usepackage{accents}
\usepackage{slashed}
\usepackage{yfonts}
\usepackage{mathrsfs,pifont}
\theoremstyle{definition}
\usepackage{tikz}
\usepackage[all]{xy}
\usepackage{mathtools}

\usepackage[style=numeric]{biblatex}
\newtheorem{Definition}{Definition}
\newtheorem{Proposition}{Proposition}
\newtheorem{Lemma}{Lemma}
\newtheorem{Corollary}{Corollary}

\newtheorem{Theorem}{Theorem}

\let\bs\boldsymbol

\def\xx{{\bs x}}
\def\aa{{\bs a}}
\def\uu{{\bs u}}
\def\dd{{\bs d}}
\def\ww{{\bs w}}
\def\zz{{\bs z}}
\def\tt{{\bs t}}
\def\zzeta{{\bs \zeta}}
\def\mmu{{\bs \mu}}

\def\sym{\text{Sym}}

\newcommand{\fC}{\mathfrak{C}}
\newcommand{\bA}{\mathsf{A}}

\newcommand{\bT}{\mathsf{T}}
\newcommand{\bK}{\mathsf{K}}

\newcommand{\Lie}{\mathrm{Lie}}
\newcommand{\Pic}{\mathrm{Pic}}

\title{3d mirror symmetry of the cotangent bundle of the full flag variety}
\author{Hunter Dinkins}
\date{}

\begin{document}

\maketitle
\begin{abstract}
Aganagic and Okounkov proved that the elliptic stable envelope provides the pole cancellation matrix for the enumerative invariants of quiver varieties known as vertex functions. This transforms a basis of a system of $q$-difference equations holomorphic in $\zz$ with poles in $\aa$ to a basis of solutions holomorphic in $\aa$ with poles in $\zz$. The resulting functions are expected to be the vertex functions of the 3d mirror dual variety. In this paper, we prove that for the cotangent bundle of the full flag variety, the functions obtained in this way recover the vertex functions for the same variety under an exchange of the parameters $\aa \leftrightarrow \zz$. As a corollary of this, we deduce the expected 3d mirror relationship for the elliptic stable envelope.
\end{abstract}

\section{Introduction}
The goal of this paper is to discuss a specific example of 3d mirror symmetry (sometimes also called symplectic duality) from the perspective of enumerative geometry. 3d mirror symmetry originated in 3-dimensional supersymmetric gauge theories. The low energy dynamics of such a theory are governed by the geometry of the moduli space of vacua. The Higgs and Coloumb branches are algebraic varieties that give two components of the moduli space of vacua, see \cite{coul2}, \cite{nakcoul}. 3d mirror symmetry expects deep relationships between the geometry of the Higgs branch of a theory and the Coloumb branch of the dual theory.

One such expectation deals with the problem of curve counting. In many cases, the Higgs branch can be constructed as a Nakajima quiver variety $X$. Using the enumerative theory of stable quasimaps to a geometric invariant theory quotient from \cite{qm}, one can package the $K$-theoretic equivariant count of genus 0 curves in $X$ into an object known as the \textit{vertex function}, see \cite{pcmilect} Section 7. The vertex function is a power series
$$
V(\zz) \in K_{\bT}(X)[[\zz]]
$$
where $\bT$ is a torus that acts on $X$ and $\zz$ is a collection of variables inserted to keep track of the degrees of quasimaps. The variables $\zz$ are usually referred to as the K\"ahler parameters of $X$. In the case where $X$ has finitely many $\bT$-fixed points, one can restrict the vertex function to the fixed points to obtain a collection of vertex functions
$$
V_p(\aa,\zz) =\sum_{\dd} c_{\dd}(\aa) \zz^{\dd} \in K_{\bT}(pt)[[\zz]], \quad p \in X^{\bT}
$$
where $\aa$ stands for the equivariant parameters of the torus $\bT$. The vertex functions are known to satisfy certain $q$-difference equations in both the variables $\aa$ and $\zz$. In some cases, the Coloumb branch $X^{!}$ can also be constructed as a Nakajima quiver variety, in which case one can ask: \textit{what is the relationship between the vertex functions and difference equations of $X$ and $X^{!}$}? 

Another side of the story involves elliptic stable envelopes. Elliptic stable envelopes were defined in \cite{AOElliptic} for Nakajima quiver varieties in order to identify the monodromy of the difference equations. Let $\textbf{Stab}(p)$ be the elliptic stable envelope of $p\in X^{\bT}$ for a certain choice of polarization and chamber (to be explained below), see \cite{AOElliptic} and \cite{SmirnovElliptic} for precise definitions. The elliptic stable envelope $\textbf{Stab}(p)$ is a section of a certain line bundle over the extended elliptic cohomology scheme $E_{\bT}(X)$. If $X^{\bT}$ is finite, then one can restrict these sections to obtain a matrix
$$
\textbf{Stab}_{p,r}:=\textbf{Stab}(p)|_{r}
$$
The entries in this matrix are certain combinations of theta functions in the parameters $\aa$ and $\zz$. 

The vertex functions $V_p(\aa,\zz)$ are known to be holomorphic in $\zz$ with poles in $\aa$. In \cite{AOElliptic}, it is proven that, for the correct choice of polarization and chamber, the matrix of restrictions of the elliptic stable envelopes provides the pole cancellation matrix of $V_p(\aa,\zz)$ and gives another set of solutions to the same system of difference equations. In other words,
\begin{equation}\label{intro1}
B_p:=\sum_{r \in X^{\bT}} \textbf{Stab}_{p,r} V_r(\aa,\zz)
\end{equation}
has no poles in a certain neighborhood of a point on a toric compactification of $\bA$, where $\bA$ is the subtorus of $\bT$ that preserves the symplectic form of $X$. Hence we can expand $B_p$ as a power series
$$
B_p=\sum_{\dd} \gamma_{\dd}(\zz) \aa^{\dd}
$$
One expectation of 3d mirror symmetry is that, under the correct normalization, (\ref{intro1}) is the vertex function of $X^{!}$. For such a statement to make sense, we must have a bijection on the fixed point sets
$$
\mathsf{b}: X^{\bT} \longleftrightarrow {X^{!}}^{\bT{!}}
$$
as well as a way of exchanging equivariant parameters with the K\"ahler parameters:
\begin{align*}
\zz &\longleftrightarrow \aa^{!} \\
    \aa &\longleftrightarrow \zz^{!}
\end{align*}

This result has been proven in the case where $X$ is a hypertoric variety in \cite{mstoric}, as well as for the vertex function at a particular fixed point for the cotangent bundle of the Grassmannian in \cite{dinkms1}. As far as we are aware, this has not been proven for any other examples.

In this paper, we study the case when $X$ is the cotangent bundle of the full flag variety in $\mathbb{C}^n$. In \cite{MirSym2}, the authors study 3d mirror symmetry from the perspective of elliptic stable envelopes and prove that in this case
$$
X^{!} \cong X
$$
i.e. the cotangent bundle of the full flag variety is 3d mirror self-symmetric. A crucial part of their argument involves the elliptic weight functions introduced in \cite{RTV}. In \cite{MirSym2}, certain combinatorial properties of the weight functions are interpreted in light of 3d mirror symmetry as a relationship between the elliptic stable envelopes of $X$ and $X^{!}$. This example has also been studied from a physical perspective in \cite{CDZ}, where 3d mirror symmetry is studied in various limits.

Here we revisit the identification $X^{!}\cong X$ from the perspective of vertex functions as explained above. Fixed points on $X$ are naturally indexed by permutations $I=(I_1,\ldots,I_n)$ of $n$. The bijection on fixed points is given by 
$$
\mathsf{b}: I \longleftrightarrow I^{-1}=:I^{!}
$$
We define the identification $\kappa$ of the equivariant and K\"ahler parameters in (\ref{kap}) below. Then our main result is

\begin{Theorem}[Theorem \ref{mainthm}]
Let $X$ be the cotangent bundle of the full flag variety and let $\left(\textbf{Stab}_{I,J}\right)_{I,J\in X^{\bT}}$ be the restriction matrix of the elliptic stable envelope for the choice of polarization and chamber given by (\ref{pol}) and (\ref{chamb}), normalized as in Definition \ref{boldstab}. Then
$$
\widetilde{V}^{!}_{I^{!}}(\aa^{!},\zz^{!})=\kappa\left(\sum_{J \in X^{\bT}} \textbf{Stab}_{I,J} \widetilde{V}_{J}(\aa,\zz) \right)
$$
where $\widetilde{V}_J(\aa,\zz)$ is the normalization of the vertex function given by Definition \ref{vnorm}.
\end{Theorem}

In other words, we prove that the elliptic stable envelope provides the transition matrix between the vertex functions of $X$ and $X^{!}\cong X$, where the roles of the equivariant and K\"ahler parameters are switched for $X^{!}$. 

Our proof relies on \cite{tRSKor} and \cite{KorZet}, in which the difference equations satisfied by the vertex functions in $\aa$ and $\zz$ are identified via the Macdonald operators (see Propositions \ref{dopzeta} and \ref{dopu} below). In \cite{NSmac}, the authors prove that solutions of these equations of a particular form are uniquely determined by their leading coefficient. In light of this, along with the results from \cite{AOElliptic}, in order to identify 
$$
\kappa\left(\sum_{J \in X^{\bT}} \textbf{Stab}_{I,J} \widetilde{V}_J(\aa,\zz) \right)
$$
with the vertex functions of $X^{!}$, it is sufficient to calculate the leading coefficient. This is done by using the known diagonal and quasiperiodicity properties of the elliptic stable envelope. Along the way, to connect our results with those of \cite{MirSym2}, we define the elliptic weight functions and identify precisely which choices of polarization and chamber are needed to relate the weight functions to the elliptic stable envelopes. 

As a corollary of our main Theorem, we obtain a simple proof of the main result of \cite{MirSym2}, which says that the inverse of the restriction matrix of the elliptic stable envelope of $X$ is equal to that of $X^{!}$ after applying $\kappa$ and appropriately permuting the rows and columns.

\subsection{Acknowledgements}
We would like to thank Andrey Smirnov for suggesting this project and for his guidance throughout.

\section{Description as a quiver variety}
\subsection{Definitions}
We construct the cotangent bundle of the full flag variety in $\mathbb{C}^n$ as a Nakajima quiver variety. Fix $n\in\mathbb{N}$. We consider the quiver with $n-1$ vertices, with dimension
$$
\mathsf{v}=(\mathsf{v}_1,\mathsf{v}_{2},\ldots,\mathsf{v}_{n-1})=(1,2,\ldots,n-1)
$$
and one framing of dimension of $n$ at vertex $n-1$.

\begin{figure}[ht]
\centering
\begin{tikzpicture}[roundnode/.style={circle,fill,inner sep=2.5pt},squarednode/.style={rectangle,fill,inner sep=3pt}] 
\node[squarednode,label=below:{$n$}](F1) at (6,-1.5){};
\node[roundnode,label=above:{1}](V1) at (1.5,0){};
\node[roundnode,label=above:{2}](V2) at (3,0){};
\node(V3) at (4.5,0){\ldots};
\node[roundnode,label=above:{$n-1$}](V4) at (6,0){};
\draw[thick, ->] (V4) -- (F1);
\draw[thick, ->] (V1) -- (V2);
\draw[thick, ->] (V2) -- (V3);
\draw[thick, ->] (V3) -- (V4);
\end{tikzpicture}
\caption{The quiver data for the cotangent bundle of the flag variety.} \label{dual}
\end{figure}
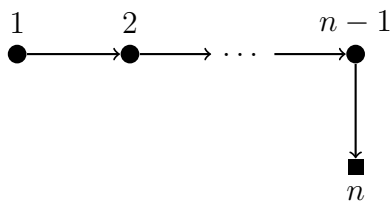
The Nakajima quiver variety associated to this data is a geometric invariant theory quotient with respect to the group
$$
G=\prod_{i=1}^{n-1} GL(V_i)
$$
where $V_i$ is a vector space of dimension $\mathsf{v}_i$. As the stability condition, we choose the $G$-character
$$
\theta: G \to \mathbb{C}^{\times}, \qquad (g_i)_i \mapsto \prod_{i=1}^{n-1} \det g_i
$$
By \cite{GinzburgLectures} Proposition 5.1.5, the $\theta$-semistable points consist of maps so that $V_{i} \to V_{i-1}$ are surjective for $i\in \{1,\ldots,n\}$, where $V_n=W$ is the framing vector space. Thus the quiver variety $X$ is the cotangent bundle of the full flag variety.

\subsection{Torus action and fixed points}
The maximal torus $\widetilde{\bA}$ of $GL(W)$ acts on $X$. After choosing a basis $\{e_i\}_{i \in \{1,\ldots,n\}}$ so that $W\cong \mathbb{C}^n$, this action is induced by the action of $\widetilde{\bA} \cong (\mathbb{C}^{\times})^n$ on $W$ given by
$$
(u_1,\ldots,u_n) \cdot (x_1,\ldots,x_n) = (u_1^{-1} x_1,\ldots,u_n^{-1} x_n)
$$
An additional torus $\mathbb{C}^{\times}_{\hbar}$ acts on $X$ by scaling the cotangent data, which is given by the maps $V_i \to V_{i-1}$ for $i\in \{1,\ldots,n\}$, by $\hbar^{-1}$. 

We define 
$$
\widetilde{\bT}:=\widetilde{\bA} \times \mathbb{C}^{\times}_{\hbar}
$$ The torus $\widetilde{\bA} \subset \widetilde{\bT}$ preserves the symplectic form of $X$ and $\widetilde{\bT}/\widetilde{\bA}$ scales it with character $\hbar$. Also, the action of $\widetilde{\bA}$ is not faithful, but has kernel $(u,u,\ldots,u)\in \widetilde{\bA}$. We write $\bA$ for the quotient by the diagonal subgroup, and denote 
$$
\bT=\bA \times \mathbb{C}^{\times}_{\hbar}
$$
Coordinates on $\bA$ will be denoted by $a_i=u_i/u_{i+1}$.

The $\bT$-fixed points on $X$ are given by data for which the vector spaces $V_i$ are spanned by coordinate vectors of $W$. Then $\ker( V_i \to V_{i-1})=\text{Span}_{\mathbb{C}}\{e_{I_i}\}$ where we assume that $V_0=0$, and the tuple $(I_1,\ldots,I_n)$ defines a permutation of $n$.

The vector spaces $V_i$ and $W$ descend to vector bundles $\mathcal{V}_i$ and $\mathcal{W}$ on $X$. It is known that $\Pic(X)\cong \mathbb{Z}^{n-1}$ is generated by the tautological line bundles
\begin{equation}\label{lb}
\mathcal{L}_i:=\det \mathcal{V}_i
\end{equation}
see \cite{kirv}. We define the K\"ahler torus 
$$
\bK:=\Pic(X) \otimes \mathbb{C}^{\times}
$$
and write $(z_1,\ldots, z_{n-1})$ for the coordinates on it. These coordinates are usually called K\"ahler parameters.

\subsection{Polarization and tangent space}
Elliptic stable envelopes were defined for Nakajima quiver varieties in \cite{AOElliptic}. The definition presented there depends on a choice of a polarization. The choice of polarization controls the $q$-periodicity of the stable envelopes, see \cite{AOElliptic} Section 3.3.

A polarization of $X$ is the choice of a $K$-theory class $T^{1/2}X$ so that the tangent bundle decomposes as
$$
TX= T^{1/2}X+\hbar^{-1} (T^{1/2}X)^{\vee} \in K_{\bT}(X)
$$
There is a natural choice of polarization for $X$, given in terms of the tautological bundles $\mathcal{V}_i$, $i=1,\ldots,n-1$ and $\mathcal{V}_n=\mathcal{W}$ by
\begin{equation}\label{pol}
T^{1/2}X = \sum_{i=1}^{n-1} \mathcal{V}_i^{\vee} \otimes \mathcal{V}_{i+1} -\sum_{i=1}^{n-1} \mathcal{V}_{i}^{\vee} \otimes \mathcal{V}_{i} \in K_{\bT}(X)
\end{equation}
In terms of the Chern roots $x^{(i)}_1,\ldots x^{(i)}_i$ of $\mathcal{V}_i$, this is given by
$$
T^{1/2}X = \sum_{i=1}^{n-1} \left( \sum_{j=1}^{i} \frac{1}{x^{(i)}_j}\right) \left( \sum_{k=1}^{i+1} x^{(i+1)}_k\right)  - \left( \sum_{j=1}^{i} \frac{1}{x^{(i)}_j}\right) \left( \sum_{k=1}^{i} x^{(i)}_k \right) \in K_{\bT}(X)
$$
At the fixed point given by a permutation $I$ of $\{1,2,\ldots, n\}$, the tangent space can be calculated by substituting for the Chern roots $x^{(i)}_j = u_{I_j}$. So in terms of the coordinates on $\widetilde{A}$,
\begin{align*}
T^{1/2}_I X &= \sum_{i=1}^{n-1} \sum_{j=1}^{i} \sum_{k=1}^{i+1} \frac{u_{I_k}}{u_{I_j}}  -  \sum_{j=1}^{i} \sum_{k=1}^{i} \frac{u_{I_k}}{u_{I_j}}  \\
&= \sum_{1\leq j < k \leq n} \frac{u_{I_k}}{u_{I_j}} \in K_{\bT}(pt)
\end{align*}
Thus
$$
T_I X =  \sum_{1\leq j < k \leq n} \frac{u_{I_k}}{u_{I_j}} + \hbar^{-1}  \sum_{1\leq j < k \leq n} \frac{u_{I_j}}{u_{I_k}} \in K_{\bT}(pt)
$$

\subsection{Chamber}
The elliptic stable envelope depends additionally on the choice of a chamber in $\Lie_{\mathbb{R}}\bA = \text{cochar}(\bA) \otimes_{\mathbb{Z}} \mathbb{R}$, where $\text{cochar}(\bA)$ denotes the cocharacter lattice of $\bA$. In the case of a variety with finitely many fixed points, the chamber controls the diagonal entries of the matrix of restrictions of the elliptic stable envelope, see \cite{AOElliptic} Section 3.3 and \cite{SmirnovElliptic} Section 2.13.

A chamber is a choice of connected component of
$$
\Lie_{\mathbb{R}}\bA -\bigcup_{w}\{\sigma \in \Lie_{\mathbb{R}}\bA \mid \langle \sigma, w \rangle=0\}
$$
where the union is taken over all $\bA$-weights of the tangent spaces at the fixed points and $\langle \cdot, \cdot \rangle$ denotes the natural pairing on characters and cocharacters. A choice of generic cocharacter $\sigma$ of $\bA$ gives a chamber $\mathfrak{C}$, and the dependence of the chamber on the cocharacter is locally constant. The tangent space at a fixed point decomposes into a direct sum of $\bA$-weight spaces
$$
T_I X= \bigoplus_{w\in Hom(\bA,\mathbb{C}^{\times})} V_I(w)
$$
A choice of chamber given by a cocharacter $\sigma$ decomposes the tangent space at a fixed point $I$ into attracting and repelling directions:
$$
T_I X= N^{+}_I+N^{-}_I
$$
where
\begin{align*}
    N_I^{+} &= \bigoplus_{\substack{w \\ \langle \sigma, w \rangle}>0} V_I(w) \\
     N_I^{-} &= \bigoplus_{\substack{w \\ \langle \sigma, w \rangle}<0} V_I(w) 
\end{align*}
Since the fixed point set is finite, every direction is either attracting or repelling. 

\begin{Definition}\label{ind}
The index bundle with respect to $\mathfrak{C}$ at a fixed point $I$, written $\text{ind}^{\mathfrak{C}}_I$ is the attracting part of the polarization restricted to $I$.
\end{Definition}

In our case of the cotangent bundle of the full flag variety, we fix once and for all the cocharacter
\begin{align}\label{chamb}
\sigma: u \mapsto (u^{-1},u^{-2},\ldots,u^{-n}), \qquad u \in \mathbb{C}^{\times}
\end{align}
and denote the corresponding chamber as $\mathfrak{C}$. With respect to this chamber, attracting weights look like $u_i/u_j$ where $i<j$, or equivalently, like monomials with positive powers in $a_i$. Explicitly, we have
$$
T_I X= N^{+}_I+N^{-}_I
$$
where
\begin{align*}
    N_I^- &=  \sum_{\substack{1\leq j < k \leq n \\ I_k > I_j}} \frac{u_{I_k}}{u_{I_j}} + \hbar^{-1}\sum_{\substack{1\leq j < k \leq n \\ I_k < I_j}} \frac{u_{I_j}}{u_{I_k}} \\
    N_I^+ &=  \sum_{\substack{1\leq j < k \leq n \\ I_k < I_j}} \frac{u_{I_k}}{u_{I_j}} + \hbar^{-1}\sum_{\substack{1\leq j < k \leq n \\ I_k > I_j}} \frac{u_{I_j}}{u_{I_k}}
\end{align*}

Given a permutation $I=(I_1,\ldots,I_n)$ of $n$, we define the ordered indices $i^{(k)}_1,\ldots,i^{(k)}_k$ so that
$$
\{i^{(k)}_1 < \ldots < i^{(k)}_k\}= \{I_1,\ldots,I_k\}
$$

\begin{Definition}\label{bruhat}
For permutations $I$ and $J$ with ordered indices $i^{(k)}_m$ and $j^{(k)}_m$, we define
\begin{equation}\label{bruhateq}
I\prec J \iff i^{(k)}_m<j^{(k)}_m \text{   for all   } k=1,\ldots,n-1 \text{  and  } m=1,\ldots,k
\end{equation}
In what follows, we will also denote by $\prec$ an arbitrary refinement of this partial order to a total order.
\end{Definition}

\subsection{Vertex functions}
The \textit{vertex function} is an important enumerative invariant of a Nakajima quiver variety $X$. The vertex function is defined as the generating function for a $K$-theoretic equivariant count of quasimaps from $\mathbb{P}^1$ to $X$. When restricted to a fixed point, the vertex function gives a power series in the K\"ahler parameters, with coefficients given by rational functions in the equivariant parameters and $q$, the coordinate on the torus $\mathbb{C}^\times_q$ which acts on the domain $\mathbb{P}^1$ of quasimaps. In the setting of elliptic cohomology, $q$ plays the role of the modular parameter of the elliptic curve $\mathbb{C}^{\times}/q^{\mathbb{Z}}$, and we assume that $0<|q|<1$.

A central property of the vertex functions is that, when normalized properly, they give a basis of solutions to a system of $q$-difference equations in both the equivariant and K\"ahler parameters, see \cite{OS}. The vertex functions are holomorphic in the K\"ahler parameters with poles in the equivariant parameters. Rather than review all the necessary theory here, we refer the reader to \cite{pcmilect} Section 7, \cite{qm}, and \cite{OkBethe} Section 1. 

For the specific case of the cotangent bundle of the full flag variety, the vertex functions have been studied in \cite{tRSKor} and \cite{KorZet}. In particular, explicit formulas were given for the vertex functions and the systems of $q$-difference equations solved by the vertex functions were identified. We review these results here.

For an indeterminate $x$, we define
$$
\varphi(x):=\prod_{i=0}^{\infty} (1-x q^i)
$$
The $q$-Pochammer symbol is given by
$$
(x)_d:=\frac{\varphi(x)}{\varphi(x q^d)}
$$
and the odd Jacobi theta function is defined by 
$$
\vartheta(x):=(x^{1/2}-x^{-1/2}) \varphi(qx) \varphi(q/x)
$$
This function satisfies 
\begin{equation}\label{theta}
\vartheta(1/x)=-\vartheta(x) \quad \text{and} \quad \vartheta(qx)= -\frac{1}{q^{1/2} x} \vartheta(x)
\end{equation}
For a vector bundle $\mathcal{V}$ on $X$, written in terms of its Chern roots as $\mathcal{V}=x_1+\ldots + x_r$, we define
$$
\Phi(\mathcal{V}):=\prod_{i=1}^{r} \varphi(x_i) \quad \text{and} \quad \Theta(\mathcal{V}):=\prod_{i=1}^{r} \vartheta(x_i)
$$
We extend this to $K_T(X)$ by
$$
\Phi(-\mathcal{V}):=\prod_{i=1}^{r} \frac{1}{\varphi(x_i)} \quad \text{and} \quad \Theta(-\mathcal{V}):=\prod_{i=1}^{r} \frac{1}{\vartheta(x_i)}
$$

% For $x=s_1+\ldots s_j-t_1-\ldots-t_k \in K_{\bT}(pt)$, where $s_1,\ldots,s_j,t_1,\ldots, t_k$ are monomials in the equivariant parameters, we define
% $$
% \Phi(x):=\frac{\prod_{i=1}^{j} \varphi(s_i)}{\prod_{i=1}^{k}\varphi(t_i)} \quad \text{and} \quad \Theta(x):=\frac{\prod_{i=1}^{j} \vartheta(s_i)}{\prod_{i=1}^{k}\vartheta(t_i)}
% $$
% For a vector bundle $\mathcal{V}\in K_{\bT}(X)$, written in terms of Chern roots as $V=x_1+\ldots +x_r$, 
The function $\Theta(\mathcal{V})$ pulls back under the elliptic Chern class map to a section of the Thom class of the bundle $\mathcal{V}$ over the elliptic cohomology scheme of $X$, see \cite{AOElliptic} Section 2.6, \cite{ell1} Section 6.1, and \cite{SmirnovElliptic} Section 2.6.

Quasimaps from $\mathbb{P}^1$ to $X$ come equipped with a notion of degree. The quasimap moduli space from which the vertex function of $X$ is defined is the disjoint union of quasimap moduli spaces for each degree. The choice of stability condition for the quiver variety determines the set of degrees for which the moduli spaces are nonempty, see \cite{pcmilect} Corollary 7.2.15. For the cotangent bundle to the full flag variety, the following definition gives the set of such degrees, see \cite{KorZet} Section 3.
\begin{Definition}
We define $C\subset \mathbb{Z}\times \mathbb{Z}^2\times \ldots \times \mathbb{Z}^{n-1}$ as the collection of integers $d_{i,j}$ where $i\in \{1,\ldots,n-1\}$ and $j\in\{1,\ldots,i\}$ so that
\begin{itemize}
    \item $d_{i,j}\geq 0$ for all $i,j$.
    \item For each $i\in\{1,\ldots,n-2\}$, there exists $\{j_1,\ldots,j_i\}\subset \{1,\ldots,i+1\}$ so that $d_{i,k}\geq d_{i+1,j_k}$ for all $k$.
\end{itemize}
\end{Definition}

% \begin{Theorem}[\cite{KorZet} Theorem 3.1]
% The vertex function of the full flag variety, restricted to a fixed point $I$, is given by the following power series:
% $$
% V_I(\uu,\zz)= e(\xx_I,\zz) \Phi((q-\hbar) T^{1/2}_I X)^{-1} \int_{C_I} \Phi((q-\hbar)T^{1/2}X) e(\xx,\zz)^{-1} d_q\xx
% $$
% where
% $$
% e(\xx,\zz)= \exp\left(\frac{1}{\ln(q)} \sum_{i=1}^{n-1} \ln(z_i) \ln(x_{i,1} \ldots x_{i,i}) \right)
% $$
% and the contour $C_I$ is equivalent to a $q$-integral where the Chern roots evaluate as
% $$
% x_{i,j}=u_{I_j} q^{-d_{i,j}}, \qquad \qquad d_{i,j}\geq 0
% $$
% The prefactor is chosen so that the power series starts with 1. 

% \end{Theorem}

\begin{Theorem}[\cite{KorZet} Theorem 3.1]\label{ver}
The vertex function of the cotangent bundle of the full flag variety, restricted to a fixed point $I$, is given by the following power series:
% \begin{multline*}
% V_p(\aa,\zz) = \sum_{d_{i,j} \in C} \zz^{\dd} \prod_{i=1}^{n-2} \prod_{j=1}^i \prod_{k=1}^{i+1}\frac{\left(\hbar \frac{x_{i+1,k}}{x_{i,j}}\right)_{d_{i,j}-d_{i+1,k}}}{\left(q \frac{x_{i+1,k}}{x_{i,j}}\right)_{d_{i,j}-d_{i+1,k}}} \\ \prod_{i=1}^{n-1}  \prod_{j,k=1}^{i}\frac{\left(q \frac{x_{i,k}}{x_{i,j}}\right)_{d_{i,j}-d_{i,k}}}{\left(\hbar \frac{x_{i,k}}{x_{i,j}}\right)_{d_{i,j}-d_{i,k}}} 
% \prod_{i=1}^n \prod_{j=1}^{n-1} \frac{\left(\hbar \frac{u_i}{x_{n-1,j}} \right)_{d_{n-1,j}}}{\left(q \frac{u_i}{x_{n-1,j}} \right)_{d_{n-1,j}}}
% \end{multline*}
\begin{multline*}
V_I(\uu,\zz) = \sum_{d_{i,j} \in C} \zz^{\dd} \prod_{i=1}^{n-2} \prod_{j=1}^i \prod_{k=1}^{i+1}\frac{\left(\hbar \frac{u_{I_k}}{u_{I_j}}\right)_{d_{i,j}-d_{i+1,k}}}{\left(q \frac{u_{I_k}}{u_{I_j}}\right)_{d_{i,j}-d_{i+1,k}}} \\ \prod_{i=1}^{n-1}  \prod_{j,k=1}^{i}\frac{\left(q \frac{u_{I_k}}{u_{I_j}}\right)_{d_{i,j}-d_{i,k}}}{\left(\hbar \frac{u_{I_k}}{u_{I_j}}\right)_{d_{i,j}-d_{i,k}}} 
\prod_{i=1}^n \prod_{j=1}^{n-1} \frac{\left(\hbar \frac{u_i}{u_{I_j}} \right)_{d_{n-1,j}}}{\left(q \frac{u_i}{u_{I_j}} \right)_{d_{n-1,j}}}
\end{multline*}

where $\zz^{\dd}=\prod_{i=1}^{n-1}\prod_{j=1}^i z_i^{d_{i,j}}$. 
% \begin{proof}
% This is equivalent to the integral form given previously. See \cite{KorZet} Theorem 3.1 for a description of the chamber. Our normalization differs from theirs by absorbing various powers of $q$ and $\hbar$ into the K\"ahler parameters.
% \end{proof}
\end{Theorem}
Note that our normalization of the vertex functions differs from that in \cite{KorZet} by absorbing various powers of $q$ and $\hbar$ into the K\"ahler parameters.

\subsection{Difference equations}
\begin{Definition}
For $1\leq r \leq n$, we define the Macdonald difference operators by
$$
D_r(\xx;q,t) = t^{r(r+1)/2 -rn}\sum_{\substack{J \subset \{1,\ldots,n\}\\|J|=r}} \prod_{\substack{i \in J \\ j \notin J}} \frac{t x_i-x_j}{x_i-x_j} \prod_{i \in J} T^{\xx}_i
$$
where $\xx=(x_1,\ldots,x_n)$ and $T^{\xx}_i: x_i \mapsto q x_i$. 
\end{Definition}
For our purposes, these difference operators will act either on the space of rational functions in $\xx$ so that none of the denominators vanish or on the space of formal power series in $\xx$ with coefficients in some ring.

We introduce new parameters $\zeta_i$ related to the K\"ahler parameters by 
$$
\frac{\zeta_{i}}{\zeta_{i+1}}=\frac{h}{q} z_i
$$
and write the vertex functions $V_{I}(\uu,\zz)$ as $V_{I}(\uu,\zzeta)$.

\begin{Definition}
For each fixed point $I \in X^{\bT}$, we define a factor
$$
\alpha_I(\uu,\zzeta,t)=\prod_{i=1}^{n} \frac{\vartheta(\zeta_i t^{i-n}) \vartheta(u_{I_i} (q/t)^{n-i})}{\vartheta(\zeta_i/u_{I_i})}
$$
\end{Definition}
\begin{Proposition}\label{alphaq}
This factor satisfies the following transformation properties:
\begin{align*}
T^{\zzeta}_i \alpha_{I}(\uu,\zzeta,t) &= \alpha_{I}(\uu,\zzeta,t) t^{n-i}u_{I_i}^{-1}T^{\zzeta}_i \\
T^{\uu}_i \alpha_{I}(\uu,\zzeta,t)&=\alpha_{I}(\uu,\zzeta,t) (q/t)^{n-I^{!}_i} \zeta_{I^{!}_i}^{-1} T^{\uu}_i 
\end{align*}
\end{Proposition}
\begin{proof}
This follows from direct computation using (\ref{theta}).
\end{proof}

For the sake of $q$-difference equations, the following function gives an equally good alternative choice
\begin{multline*}
\alpha_{I}(\uu,\zzeta,t) =\exp\left(-\frac{1}{\ln(q)} \sum_{i=1}^{n-1} \ln(z_i) \ln( \mathcal{L}_i|_{I} t^{i(i+1)/2-in}) \right) \\ \exp \left(-\frac{1}{\ln(q)} \ln(\zeta_n) \ln(u_1\ldots u_n t^{n(1-n)/2}) \right)
\end{multline*}

\begin{Definition}\label{vnorm}
We defined normalized vertex functions by
$$
\widetilde{V}_I(\uu,\zzeta)= \alpha_I(\zzeta,\uu,\hbar) \Phi((q-\hbar) T^{1/2}_I X)  V_{I}(\uu,\zzeta)
$$
\end{Definition}

\begin{Proposition}\label{dopzeta}[\cite{tRSKor} Theorem 2.6]
% (CONFIMRED ON MAPLE) 
For all fixed points $I$, the normalized vertex function is an eigenvector of the Macdonald operators:
$$
D_r(\zzeta;q,\hbar) \widetilde{V}_{I}(\uu,\zzeta) = e_r(\uu^{-1})  \widetilde{V}_{I}(\uu,\zzeta)
$$
where $e_r(\uu^{-1})$ denotes the $r$th elementary symmetric polynomial in $\uu^{-1}$.
\end{Proposition}
\begin{proof}
Observe that the term $\alpha_I(\uu,\zzeta,\hbar)$ has the same $q$-periodicity in $\zzeta$ as the prefactor in Theorem 2.6 of \cite{tRSKor}. Furthermore, the term $\Phi((q-\hbar) T^{1/2}_I X)$ only involves the equivariant parameters, so it does not affect the $q$-difference properties in the variables $\zzeta$. Now the result follows from Theorem 2.6 in \cite{tRSKor}.
\end{proof}
Similarly, we have
\begin{Proposition}\label{dopu}[\cite{KorZet} Theorem 4.8]
% (CONFIRMED ON MAPLE) 
For all fixed points $I$, the normalized vertex function is an eigenvector of the Macdonald operators:
$$
(q/\hbar)^{r(n-1) } D_r(\uu;q,q/\hbar) \widetilde{V}_I(\uu,\zzeta) = e_r(\zzeta^{-1}) \widetilde{V}_I(\uu,\zzeta)
$$
where $e_r(\zzeta^{-1})$ denotes the $r$th elementary symetric polynomial in $\zzeta^{-1}$.
\end{Proposition}

\subsection{The dual variety}
We denote by $X^{!}$ another copy of the cotangent bundle of the full flag variety, constructed as a Nakajima quiver variety in the same way as $X$ above. In particular, we assume the choice of stability $\theta^{!}$ is the same.

This variety is equipped with the action of a torus $\widetilde{\bT}^{!}$, with quotient torus $\bT^{!}$. We write $u_1^{!},\ldots,u_n^{!},\hbar^{!}$ for the coordinates on $\widetilde{\bT}^{!}$ and $a_1^{!},\ldots,a_{n-1}^{!},\hbar^{!}$ for the coordinates on $\bT^{!}$. We choose the same chamber as for $X$, and denote a generic cocharacter in the chamber by $\sigma^{!}$. As before, the choice of chamber provides a decomposition of the tangent bundle at a fixed point $I^{!}$
$$
T_{I^{!}} X^{!}=N_{I^{!}}^{!-}+N_{I^{!}}^{!+} \in K_{\bT^{!}}(pt)
$$
into attracting and repelling directions.

Similarly, the K\"ahler torus of $X^{!}$ has coordinates $z_1^{!},\ldots,z_{n-1}^{!}$. As discussed in the introduction, 3d mirror symmetry expects the existence of a bijection
$$
X^{\bT} \longleftrightarrow {X^{!}}^{\bT^{!}}
$$
and an isomorphism of tori
$$
\kappa: \bT \times \bK \times \mathbb{C}^{\times}_q \to \bT^{!} \times \bK^{!} \times \mathbb{C}^{\times}_q
$$
In our context, we define the bijection on fixed points as
$$
I \longleftrightarrow I^{-1}
$$
For uniformity in our formulas below, we prefer to write $I^{!}$ for the inverse permutation $I^{-1}$. We define $\kappa$ by 
\begin{align}\label{kap} \nonumber
    z_{i} &\mapsto \hbar^{!} a^{!}_i \\ \nonumber
    a_i &\mapsto \frac{\hbar^{!}}{q} z^{!}_i \\ \nonumber
    \hbar &\mapsto \frac{q}{\hbar^{!}} \\
    q &\mapsto q
\end{align}
As in the case of $X$, we define new parameters $\zeta^{!}_1,\ldots, \zeta^{!}_n$ related to $z_i^{!}$ by 
$$
\frac{\zeta^{!}_i}{\zeta^{!}_{i+1}}=\frac{\hbar^{!}}{q} z^{!}_i
$$
The parameters $\zeta_i$ and $\zeta_i^{!}$ can be thought of as coordinates on extensions $\widetilde{\bK}$ and $\widetilde{\bK}^{!}$ of the K\"ahler tori. In this language, the map $\kappa$ is induced by the map
$$
\widetilde{\bT}\times \widetilde{\bK} \times \mathbb{C}^{\times}_q \to \widetilde{\bT}^{!} \times \widetilde{\bK}^{!} \times \mathbb{C}^{\times}_q
$$
given by
\begin{align*}
    \zeta_{i} &\mapsto \zeta^{!}_i \\
    u_i &\mapsto u_i^{!} \\
    \hbar &\mapsto \frac{q}{\hbar^{!}} \\
    q &\mapsto q
\end{align*}
Abusing notation, we will also write this map as $\kappa$ when it appears in what follows.

From the explicit form of $\kappa$, it is easy to verify that the differential of $\kappa$, restricted to suitable subtori, satisfies
\begin{equation}\label{stabch}
d \kappa(\sigma)= \theta^{!} \quad \text{and} \quad d \kappa(\theta)=\sigma^{!}
\end{equation}
This is an expected property of 3d mirror symmetry, see \cite{KS2} where the property (\ref{stabch}) is part of the definition of 3d mirror symmetry.

\subsection{Limits of vertex functions}
Given a choice of chamber $\mathfrak{C}$ determined by a cocharacter $\sigma: \mathbb{C}^{\times} \to \bA$, we define
\begin{equation}\label{0c}
V_p(0_{\mathfrak{C}},\zz):= \lim_{w \to 0} V_p(\sigma(w),\zz)
\end{equation}
Since all equivariant parameters appear in terms of the form
$$
\frac{1-w u_i/u_j}{1-w' u_i/u_j}
$$
for some $w,w' \in \mathbb{Q}(q,\hbar)$, this limit is a well-defined element of $\mathbb{C}(q,\hbar)[[\zz]]$.

For our choice of chamber in (\ref{chamb}), we have
\begin{Proposition}\label{limver}
% For the chamber chosen above, we have (weights $a_i$ go to 0)
$$
\kappa\left(V_I(0_{\mathfrak{C}},\zz)\right)= \Phi((q-\hbar^{!}) N_{I^{!}}^{!+})
$$
\end{Proposition}
\begin{proof}
Recall that attracting directions of the tangent space look like $u_i/u_j$ for $i<j$. In the limit of the vertex function with respect to the chamber $\mathfrak{C}$, these weights contribute terms of the form
$$
\lim_{a\to 0} \frac{1-w a}{1-w' a} = 1
$$
for $w,w' \in \mathbb{Q}(q,\hbar)$. Repelling weights contribute terms of the form
$$
\lim_{a\to \infty} \frac{1-w a}{1-w'a} = \frac{w}{w'}
$$
So if $a$ is an repelling weight, we obtain contributions of the form
$$
\lim_{a\to\infty} \frac{\left(\hbar a \right)_d}{\left(q a \right)_d} = \left(\frac{\hbar}{q}\right)^{d}
$$
Putting this together with Theorem \ref{ver}, we find that
% \begin{multline*}
%     V_p(0_{\mathfrak{C}},\zz) = \sum_{d_{i,j} \in C} \zz^{\dd} \prod_{i=1}^{n-2} \prod_{j=1}^{i} \prod_{\substack{k=1\\ I_k>I_j}}^{i+1} \left( \frac{\hbar}{q}\right)^{d_{i,j}-d_{i+1,k}} \prod_{j=1}^i \frac{(\hbar)_{d_{i,j}-d_{i+1,j}}}{(q)_{d_{i,j}-d_{i+1,j}}} \\ \prod_{i=1}^{n-1} \prod_{\substack{j,k=1 \\ I_k>I_j}}^{i} \left( \frac{q}{\hbar}\right)^{d_{i,j}-d_{i,k}}  \prod_{i=1}^{n} \prod_{\substack{j=1 \\ i>I_j}}^{n-1} \left( \frac{\hbar}{q}\right)^{d_{n-1,j}} \prod_{j=1}^{n-1} \frac{(\hbar)_{d_{n-1,j}}}{(q)_{d_{n-1,j}}}
% \end{multline*}
% Rearranging gives
\begin{multline*}
    V_p(0_{\mathfrak{C}},\zz) = \sum_{d_{i,j} \in C_I} \zz^{\dd} \prod_{i=1}^{n-1} \frac{(\hbar)_{d_{n-1,i}}}{(q)_{d_{n-1,i}}} \prod_{i=1}^{n-2} \prod_{j=1}^i  \frac{(\hbar)_{d_{i,j}-d_{i+1,j}}}{(q)_{d_{i,j}-d_{i+1,j}}} \\ \prod_{i=1}^{n-2} \prod_{j=1}^{i} \prod_{\substack{k=1\\ I_k>I_j}}^{i+1} \left( \frac{\hbar}{q}\right)^{d_{i,j}-d_{i+1,k}}  \prod_{i=1}^{n-1} \prod_{\substack{j,k=1 \\ I_k>I_j}}^{i} \left( \frac{q}{\hbar}\right)^{d_{i,j}-d_{i,k}}  \prod_{i=1}^{n} \prod_{\substack{j=1 \\ i>I_j}}^{n-1} \left( \frac{\hbar}{q}\right)^{d_{n-1,j}} 
\end{multline*}
Now, define indices $f_{j,k}$ so that 
$$
f_{j,k} = \begin{cases} d_{k-1,j}-d_{k,j} & k<n \\ d_{n-1,j} & k=n \end{cases}
$$
We observe that
\begin{align*}
\prod_{1 \leq j < k \leq n} (z_j \ldots z_{k-1})^{f_{j,k}} &= 
% \prod_{i=1}^{n-1} z_i^{\sum_{j=1}^i d_{n-1,j} + \sum_{k=i+1}^{n-1} d_{k-1,j}-d_{k,j}} 
% \\
% &= \prod_{i=1}^{n-1} z_i^{\sum_{j=1}^{i} d_{i,j}} 
\zz^{\dd}
\end{align*}
Furthermore, the $q$-Pochammer terms in $V_p(0_{\mathfrak{C}},\zz)$ give
$$
\prod_{i=1}^{n-1}\frac{(\hbar)_{d_{n-1,i}}}{(q)_{d_{n-1,i}}}  \prod_{i=1}^{n-2} \prod_{j=1}^{i} \frac{(\hbar)_{d_{i,j}-d_{i+1,j}}}{(q)_{d_{i,j}-d_{i+1,j}}} = \prod_{1\leq j < k \leq n}\frac{(\hbar)_{f_{j,k}}}{(q)_{f_{j,k}}}
$$
% Rewriting the powers of $\hbar/q$ gives
% \begin{align*}
%     &\prod_{i=1}^{n-1} \left( \prod_{\substack{1 \leq j < k \leq i \\ I_j< I_k}} \left(\frac{\hbar}{q}\right)^{f_{j,i}} \prod_{\substack{1 \leq j < k \leq i-1 \\ I_j> I_k}} \left(\frac{\hbar}{q}\right)^{f_{k,i}}  \prod_{\substack{j=1 \\ I_j> I_i}}^{i-1} \left( \frac{\hbar}{q}\right)^{d_{i,j}-d_{i,i}}\right) \left( \frac{\hbar}{q} \right)^{f_{i,n}(n-I_i)} \\
%     &=\prod_{i=1}^{n-1} ( \prod_{j=1}^{i-1} \left( \frac{\hbar}{q} \right)^{f_{j,i} \text{above j at most i that stay mapped above Ij}} \prod_{j=1}^{i-2} \left( \frac{\hbar}{q}\right)^{f_{j,i} \text{below j that are mapped above Ij}} \\ 
%     &\prod_{\substack{j=1 \\I_j>I_i}}^{i-1} \left(\frac{\hbar}{q}\right)^{\sum_{l=j}^{i-1} f_{l,i}}) \left( \frac{\hbar}{q} \right)^{f_{i,n}(n-I_i)}  \\
%     &=\prod_{i=1}^{n-1} ( \left(\frac{\hbar}{q}\right)^{f_{i-1,i}\delta(I_{i-1}<I_i)} \prod_{j=1}^{i-2} \left( \frac{\hbar}{q} \right)^{f_{j,i} \text{everything at most i that is mapped above Ij}}\\ 
%     &\prod_{\substack{j=1 \\I_j>I_i}}^{i-1} \left(\frac{\hbar}{q}\right)^{\sum_{l=j}^{i-1} f_{l,i}}) \left( \frac{\hbar}{q} \right)^{f_{i,n}(n-I_i)} 
% \end{align*}

A proof by induction on the order $\prec$ from Definition \ref{bruhat} can show that the rest of the terms combine to give in total
$$
\prod_{1 \leq j <k \leq n} \frac{\varphi\left( q \left(\frac{\hbar}{q} \right)^{k-j+\delta(I_j<I_k)} z_j \ldots z_{k-1}\right)}{\varphi\left( \left(\frac{\hbar}{q} \right)^{k-j-1+\delta(I_j<I_k)} z_j \ldots z_{k-1}\right)}
$$
where $\delta(a<b)$ is 1 if $a<b$ and 0 otherwise. Applying the map $\kappa$, this is clearly seen to be $\Phi((q-\hbar^{!}) N_{I^!}^+)$.

\end{proof}

\section{Weight functions and elliptic stable envelopes}
\subsection{Definitions}
Following \cite{RTV}, we define elliptic weight functions associated to the cotangent bundle of the full flag variety. We will show that these coincide with the elliptic stable envelope of $X$. For a further discussion of elliptic weight functions and their relation to stable envelopes in the case of the cotangent bundle of the Grassmannian, see \cite{ellqg}.

The weight functions depend on the parameters
\begin{itemize}
    \item $w_j$ for $1\leq j \leq n$, which we abbreviate by $\ww$.
     \item $t^{(k)}_j $ for $1\leq k \leq n$ and $1\leq j \leq k$, where $t^{(n)}_j=w_j$. We abbreviate by $\tt^{(k)}$ the variables $t^{(k)}_1,\ldots,t^{(k)}_k$ and by $\tt$ the variables $\tt^{(1)},\ldots,\tt^{(n-1)}$.
    \item $\mu_j$ for $1\leq j \leq n$, which we abbreviate by $\mmu$.
    \item $\hbar$
\end{itemize}
We will identify these with the parameters of $X$ in (\ref{wtparam}) below after discussing the main properties of the weight functions.

Let
$$
U_{I}(\tt,\ww,\hbar,\mu)= \prod_{k=1}^{n-1}\left( \prod_{a=1}^{k} \prod_{c=1}^{k+1} \psi_{I,k,a,c}(t^{(k+1)}_c/t^{(k)}_a) \right) \prod_{b=a+1}^{k} \frac{\vartheta(\hbar t^{(k)}_b/t^{(k)}_a)}{\vartheta(t^{(k)}_b/t^{(k)}_a )}
$$
where 
$$
\psi_{I,k,a,c}(x) = \begin{dcases} \vartheta( x) & i^{(k+1)}_c<i^{(k)}_a \\
\frac{\vartheta\left(x \hbar^{-\delta(I_{k+1}<i^{(k)}_a)} \mu_{j(I,k,a)}/\mu_{k+1} \right)}{\vartheta\left( \hbar^{-\delta(I_{k+1}<i^{(k)}_a)} \mu_{j(I,k,a)}/\mu_{k+1} \right)} & i^{(k+1)}_c=i^{(k)}_a  \\
\vartheta(x/\hbar) & i^{(k+1)}_c>i^{(k)}_a 
\end{dcases}
$$
Here, the index $j(I,k,a)\in \{1,\ldots,n\}$ is the index so that 
$$
I_{j(I,k,a)}=i^{(k)}_a
$$
and
$$
\delta(a<b) = \begin{cases}
1 & a<b \\
0 & \text{otherwise}
\end{cases}
$$
Let 
$$
E(\tt,\hbar)= \prod_{k=1}^{n-1} \prod_{a,b=1}^{k} \vartheta(\hbar t^{(k)}_b/t^{(k)}_a)
$$
and
\begin{equation*}
\widetilde{U}_I(\tt,\ww,\hbar,\mmu)=\frac{U_{I}(\tt,\ww,\hbar,\mmu)}{E(\tt,\hbar)} 
% \\
% = \vartheta(\hbar)^{-n(n-1)/2} \prod_{k=1}^{n-1}\left( \prod_{a=1}^{k} \prod_{c=1}^{k+1} \psi_{I,k,a,c}(t^{(k+1)}_c/t^{(k)}_a) \right) \prod_{b=a+1}^{k} \frac{1}{\vartheta(\hbar t^{(k)}_a/t^{(k)}_b) \vartheta(t^{(k)}_b/t^{(k)}_a )}
\end{equation*}
Define the symmetrization of a function of $\tt^{(k)}$ by
$$
\sym_{t^{(k)}} f(t^{(k)}_1,\ldots,t^{(k)}_k) = \sum_{\tau \in S_{k}} f(t^{(k)}_{\tau(1)},\ldots,t^{(k)}_{\tau(k)})
$$
Define
$$
W_I(\tt,\ww,\hbar,\mmu)=  \vartheta(\hbar^{-1})^{n(n-1)/2} \sym_{\tt^{(1)}} \ldots \sym_{\tt^{n-1}} U_I(\tt,\ww,\hbar,\mmu)
$$
$W_I(\tt,\zz,\hbar,\mmu)$ is known as the elliptic weight function, see \cite{RTV}. We will need the normalized weight function
\begin{align*}
\widetilde{W}_I(\tt,\ww,\hbar,\mmu)&= \frac{W_{I}(\tt,\ww,\hbar,\mmu)}{E(\tt,\hbar)} \\
&= \vartheta(\hbar^{-1})^{n(n-1)/2} \sym_{\tt^{(1)}} \ldots \sym_{\tt^{n-1}} \widetilde{U}_I(\tt,\ww,\hbar,\mmu)
\end{align*}
\subsection{Properties of the weight functions}
Although our normalization differs slightly from that in \cite{RTV}, the proofs of all of the properties of the weight functions stated below can be obtained from straightforward modifications of the proofs of the analogous properties in \cite{RTV} Section 2. For a function $f(\tt)$, let $f(\ww_I)$ denote the result of substituting $t^{(k)}_j=w_{i^{(k)}_j}$ in $f$. 
\begin{Lemma}[\cite{RTV} Lemma 2.4]\label{stabt}
$\widetilde{W}_{I}(\ww_J,\ww,\hbar,\mmu)=0$ unless $I\prec J$
\end{Lemma}
Next, let
$$
P_I(\ww,\hbar)=  \prod_{\substack{k<l \\ I_l<I_k}}\vartheta(w_{I_k}/w_{I_l}) \prod_{\substack{k<l \\ I_k<I_l}}\vartheta(\hbar w_{I_k}/w_{I_l})
$$

\begin{Lemma}[\cite{RTV} Lemma 2.5]\label{rtv2}
$$
\widetilde{W}_I(\ww_{I},\ww,\hbar,\mmu)= P_I(\ww,\hbar)
$$
\end{Lemma}

Define the functions
% \begin{align*}
%   G(\tt,\ww,\hbar,\mmu)= \prod_{k=1}^{n-1} \prod_{a=1}^{k} \prod_{c=1}^{k+1} \vartheta(t^{(k)}_a/t^{(k+1)}_c) \prod_{k=1}^{n-1}\frac{\vartheta(\hbar^{-1} \mu_k/\mu_{k+1} 1/(t^{(k)}_1 \ldots t^{(k)}_k))}{\vartheta(\hbar^{-1} \mu_k/\mu_{k+1})\vartheta(1/(t^{(k)}_1 \ldots t^{(k)}_k))}
% \end{align*}
% and 
% \begin{align*}
%   G_I(\ww,\hbar,\mmu) =\theta(\hbar)^{\text{rk} \text{ind}_I} \frac{\vartheta(\hbar^{-1})\vartheta(\det \text{ind}_I)}{\vartheta(\hbar^{-1} \det \text{ind}_I)} \prod_{k=1}^{n-1}\frac{\vartheta(1/(z_{I_1} \ldots z_{I_k})) \vartheta(\hbar^{-1} \mu_j/\mu_{j+1})}{\vartheta(1/(z_{I_1} \ldots z_{I_k}) \hbar^{-1} \mu_j/\mu_{j+1})}
% \end{align*}
% and
% $$
% G_I(\tt,\ww,\hbar,\mmu)=G(\tt,\ww,\hbar,\mmu)G_I(\zz,\hbar,\mmu)
% $$
% Instead, we could also use the functions
\begin{align*}
  G(\tt,\ww,\hbar,\mmu)= \prod_{k=1}^{n-1} \prod_{a=1}^{k} \prod_{c=1}^{k+1} \vartheta(t^{(k+1)}_c/t^{(k)}_a) \prod_{k=1}^{n-1}\frac{\vartheta(\hbar \mu_{k+1}/\mu_{k} t^{(k)}_1 \ldots t^{(k)}_k)}{\vartheta(\hbar \mu_{k+1}/\mu_{k})\vartheta(t^{(k)}_1 \ldots t^{(k)}_k)}
\end{align*}
and 
\begin{multline*}
  G_I(\ww,\hbar,\mmu)  = \\ \vartheta(\hbar)^{N(I)} \frac{\vartheta(\hbar)\vartheta( \prod_{\substack{1\leq j <k\leq n \\ I_j<I_k}} w_{I_k}/w_{I_j})}{\vartheta(\hbar \prod_{\substack{1\leq j <k\leq n \\ I_j<I_k}} w_{I_k}/w_{I_j})} \prod_{k=1}^{n-1}\frac{\vartheta(w_{I_1} \ldots w_{I_k}) \vartheta(\hbar \mu_{k+1}/\mu_{k})}{\vartheta(w_{I_1} \ldots w_{I_k} \hbar \mu_{k+1}/\mu_{k})}
\end{multline*}
where $N(I)$ is the number of non-inversions of $I$:
$$
N(I):= \sum_{\substack{1\leq j < k \leq n \\ I_j<I_k}} 1 = \frac{n(n-1)}{2} - \sum_{\substack{1\leq j < k \leq n \\ I_j>I_k}} 1
$$
Let 
$$
G_I(\tt,\ww,\hbar,\mmu)= G(\tt,\ww,\hbar,\mmu) G_{I}(\ww,\hbar,\mmu)
$$
\begin{Lemma}[\cite{RTV} Lemma 2.3]
The ratio $W_I(\tt,\ww,\hbar,\mmu)/G_I(\tt,\ww,\hbar,\mmu)$ does not change when any of the variables $\tt$, $\ww$, $\mmu$, and $\hbar$ are shifted by $q$.
\end{Lemma}
As a result, we deduce:
\begin{Lemma}\label{rtv3}
$\widetilde{W}_I(\tt,\ww,\hbar,\mmu)$ has the same transformation properties as $$G_I(\tt,\ww,\hbar,\mmu)/E(\tt,\hbar)$$ under shifts of the variables $\tt$, $\ww$, $\mmu$, and $\hbar$ by $q$.
\end{Lemma}

\subsection{Elliptic stable envelopes}
We identify the Chern roots, equivariant parameters, and K\"ahler parameters of $X$ with the weight function parameters by 
\begin{align} \label{wtparam} \nonumber
    \hbar &\mapsto \hbar \\ \nonumber
    t^{(k)}_a &\mapsto 1/x^{(k)}_a \\ \nonumber
    \mu_j/\mu_{j+1} &\mapsto \hbar z_i \\
    w_i &\mapsto 1/u_{i}
\end{align}
Under this identification, Lemmas \ref{stabt}, \ref{rtv2}, and \ref{rtv3} give us the following three lemmas.

\begin{Lemma}
$\widetilde{W}_{I}(\xx_J,\uu,\hbar,\zz)=0$ unless $I\prec J$.
\end{Lemma}

\begin{Lemma}\label{stabq}
The function $\widetilde{W}_I(\xx,\uu,\hbar,\zz)$ has the same transformation properties as
\begin{align*}
%  \prod_{k=1}^{n-1} \prod_{a,b=1}^{k} \frac{1}{\vartheta(\hbar x^{(k)}_b/x^{(k)}_a)}  \prod_{k=1}^{n-1} \prod_{a=1}^{k} \prod_{c=1}^{k+1} \vartheta(x^{(k+1)}_c/x^{(k)}_a) \prod_{k=1}^{n-1}\frac{\vartheta(z_k x^{(k)}_1 \ldots x^{(k)}_k)}{\vartheta(z_k)\vartheta(x^{(k)}_1 \ldots x^{(k)}_k)} \\
%  \theta(\hbar)^{\text{rk} \text{ind}_I} \frac{\vartheta(\hbar^{-1})\vartheta(\det \text{ind}_I)}{\vartheta(\hbar^{-1} \det \text{ind}_I)} \prod_{k=1}^{n-1}\frac{\vartheta(u_{I_1} \ldots u_{I_k}) \vartheta(z_k)}{\vartheta(u_{I_1} \ldots u_{I_k} z_k )} \\ =
 \Theta(\hbar)^{\text{rk}(\text{ind}^{-\mathfrak{C}}_I)} \Theta(T^{1/2} X) \frac{\vartheta(\hbar^{-1})\vartheta(\det \text{ind}^{-\mathfrak{C}}_I)}{\vartheta(\hbar^{-1} \det \text{ind}^{-\mathfrak{C}}_I)} \prod_{k=1}^{n-1} \frac{\vartheta(z_k \mathcal{L}_k)}{\vartheta(z_k) \vartheta(\mathcal{L}_k)} \prod_{k=1}^{n-1} \frac{\vartheta(\mathcal{L}_k|_I)\vartheta(z_k)}{\vartheta(\mathcal{L}_k|_I z_k)}
\end{align*}
under shifts of the variables $\xx,\uu,\zz,$ and $\hbar$ by $q$. Here, $\mathcal{L}_i$ denotes the tautological line bundle from (\ref{lb}) and $\text{ind}_{I}^{-\mathfrak{C}}$ denotes the index bundle from Definition \ref{ind}.
\end{Lemma}
For a function $f(\xx)$, we write $f(\xx_I)$ for the substitution $x^{(k)}_j=u_{I_j}$ into $f$. Then
\begin{Lemma}\label{stabd}
\begin{align*}
\widetilde{W}_I(\xx_I,\uu,\hbar,\zz) 
% &=  \prod_{\substack{k<l \\ I_l<I_k}}\vartheta(u_{I_l}/u_{I_k}) \prod_{\substack{k<l \\ I_k<I_l}}\vartheta(\hbar u_{I_l}/u_{I_k})\\
% &=(-1)^{n} (-1)^I \prod_{\substack{k<l \\ I_l<I_k}}\vartheta(u_{I_l}/u_{I_k}) \prod_{\substack{k<l \\ I_k<I_l}}\vartheta(\hbar^{-1} u_{I_k}/u_{I_l}) \\
&= (-1)^{n(n-1)/2} (-1)^I \Theta(N_I^{+})
\end{align*}
where $(-1)^I$ is the sign of the permutation $I$.
\end{Lemma}

Combining these facts gives 
\begin{Theorem}
Up to a sign, as sections of a line bundle over the elliptic cohomology scheme of $X$, the weight function coincides with the elliptic stable envelope:
$$
\widetilde{W}_I(\xx,\uu,\hbar,\zz) = (-1)^{n(n-1)/2} (-1)^I \text{Stab}_{-\mathfrak{C},T^{1/2}X}(I)
$$
\end{Theorem}
\begin{proof}
Recall that elliptic cohomology classes are defined as sections of a particular line bundle over the elliptic cohomology scheme. In \cite{AOElliptic}, the elliptic stable envelope is defined as the unique elliptic cohomology class satisfying two properties. The first property is a support condition, which says that $\text{Stab}_{-\mathfrak{C},T^{1/2}X}(I)$ is supported on the full attracting set of $I$ with respect to the cocharacter $-\sigma$, see Section 3.3.5 of \cite{AOElliptic}. A simple calculation shows that the order $\prec$ of Definition \ref{bruhat} coincides with the partial order by attraction with respect to $\sigma$. The second condition is an explicit identification of the restrictions $\text{Stab}_{-\mathfrak{C},T^{1/2}X}(I)|_{I}$. As explained in \cite{SmirnovElliptic} Section 2.13, this condition requires that 
$$
\text{Stab}_{-\mathfrak{C},T^{1/2}X}(I)|_{I} = \Theta(N_I^{+})
$$

Lemma \ref{stabq} means that $\widetilde{W}_I(\xx,\uu,\hbar,\zz)$ is a section of the right line bundle over the elliptic cohomology scheme of $X$. In other words, it is actually an elliptic cohomology class.

Lemma \ref{stabt} is precisely the support condition.

After multiplication by $(-1)^{n(n-1)/2}(-1)^I$, Lemma \ref{stabd} is precisely the condition on the restriction $\text{Stab}_{-\mathfrak{C},T^{1/2}X}(I)|_{I}$.
\end{proof}
In what follows, we write 
$$
\text{Stab}_{I,J}=\text{Stab}_{-\fC,T^{1/2}X}(I)|_J
$$
with the choice of chamber and polarization understood.

\section{3d mirror symmetry of vertex functions}

\begin{Definition}\label{boldstab}
We define a new normalization of the elliptic stable envelope by
\begin{multline*}
\textbf{Stab}_{I,J}=  \sqrt{\frac{\det T_I^{1/2} X \kappa^{-1}( \det N_{I^{!}}^{!+} )}{\det N_I^+ \kappa^{-1}( \det T_{I^{!}}^{1/2} X^{!})}} \frac{\kappa^{-1}\left(\alpha_I^{!}(\uu^{!},\zzeta^{!},\hbar^{!})\right)}{\alpha_J(\uu,\zzeta,\hbar)} \\ \frac{\text{Stab}_{I,J}}{\kappa^{-1}\left(\Theta(N^{!+}_{I^{!}})\right)} \frac{\kappa^{-1}\left(\Theta(T^{1/2}_{I^{!}}X^{!})\right)}{\Theta(T^{1/2}_J X)}  
\end{multline*}
\end{Definition}

Our main result is:
\begin{Theorem} \label{mainthm}
$$
\widetilde{V}^{!}_{I^{!}}(\aa^{!},\zz^{!})=\kappa\left(\sum_{J \in X^{\bT}} \textbf{Stab}_{I,J} \widetilde{V}_{J}(\aa,\zz) \right)
$$
We refer to this by saying that the elliptic stable envelope is the transition matrix between the vertex functions of $X$ and the vertex functions of $X^{!}$.
\end{Theorem}

Alternatively, by canceling the repeated transcendental factors in $\textbf{Stab}_{I,J}$ and $\widetilde{V}_J(\aa,\zz)$ and using Lemma \ref{lem} below, this is equivalent to
\begin{Theorem}\label{mainthm2}
$$
\overline{V}^{!}_{I^{!}}(\aa^{!},\zz^{!})=\kappa\left(\sum_{J \in X^{\bT}} \overline{\textbf{Stab}}_{I,J} \overline{V}_{J}(\aa,\zz) \right)
$$
where $\overline{V}_{I}(\aa,\zz)= \Phi((q-\hbar)N_I^{+}) V_I(\aa,\zz)$ and 
$$
\overline{\textbf{Stab}}_{I,J}=\sqrt{\frac{\det T^{1/2}_I X \det N_J^{+}}{\det T_J^{1/2}X \det N_I^+}} \frac{\text{Stab}_{I,J}}{\Theta(N_J^{+})}
$$ 
\end{Theorem}

The formulation of Theorem \ref{mainthm} reflects our preferred normalization of the vertex functions and is particularly transparent for analyzing the $q$-difference properties. On the other hand, Theorem \ref{mainthm2} better reflects the pole cancellation properties of the elliptic stable envelope, see the discussion in the proof of Theorem \ref{mainthm} below.

Our goal in the remainder of this paper is to prove this theorem. From Propositions \ref{dopzeta} and \ref{dopu}, we know that the vertex functions are eigenvectors of the Macdonald operators. In \cite{NSmac}, it is shown that such functions are essentially uniquely determined by their leading coefficient. In what follows, we will use the known properties of the elliptic stable envelope to compute the leading term of the right hand side of Theorem \ref{mainthm}. It will be apparent that it agrees with the leading term of the left hand side.

\subsection{Various normalizations}
As a function of the Chern roots of the tautological bundles and of the K\"ahler parameters, let
$$
e(\xx,\zz):=\prod_{i=1}^{n-1}\frac{\vartheta(\mathcal{L}_i)\vartheta(z_i)}{\vartheta(\mathcal{L}_i z_i)}
$$
This function transforms as follows:
\begin{align} \label{eqp} \nonumber
    &e(\xx,\zz)|_{x^{(i)}_j =q x^{(i)}_j} = e(\xx,\zz) z_i \\
    &e(\xx,\zz)|_{z_i=qz_i} = e(\xx,\zz) \mathcal{L}_i
\end{align}
From the perspective of $q$-difference operators in $\uu$, the function $e(\xx_I,\zz)$ is equivalent to 
$$
\alpha_{I}(\uu,\zzeta,\hbar)  \exp(-\frac{\ln(\zeta_n)\ln(u_1\ldots u_n)}{\ln(q)})
$$
\begin{Definition}
We define a further normalization of the elliptic stable envelope by
$$
S_{I,J} = e(\xx_I,\zz)^{-1}  \Theta(N_I^+)^{-1} \text{Stab}_{I,J} e(\xx_J,\zz)
$$
\end{Definition}
\begin{Lemma}\label{S}
 $S_{I,J}$ is $q$-periodic with respect to shifts of $\uu$ and $\zzeta$. Furthermore, $S_{I,I}=1$.
\end{Lemma}
\begin{proof}
The first claim follows from the known $q$-quasiperiodicity of the various functions from Lemma \ref{stabq}, Proposition \ref{alphaq}, and (\ref{eqp}). The second claim follows from Lemma \ref{stabd}.
\end{proof}

Tracing back the definitions shows that $\textbf{Stab}_{I,J}$ is related to $S_{I,J}$ by
\begin{multline}\label{stabs}
\textbf{Stab}_{I,J}=  \sqrt{\frac{\det T_I^{1/2} X \kappa^{-1}( \det N_{I^{!}}^{!+} X^{!})}{\det N_I^+ \kappa^{-1}( \det T_{I^{!}}^{1/2} X^{!})}} \frac{\kappa^{-1}\left(\alpha_I^{!}(\uu^{!},\zzeta^{!},\hbar^{!})\right)}{\alpha_J(\uu,\zzeta,\hbar)}  \\ \frac{\kappa^{-1}\left(\Theta(T^{1/2}_{I^{!}}X^{!})\right)}{\Theta(T^{1/2}_J X)}  \frac{e(\xx_I,\zz)}{e(\xx_J,\zz)} \frac{\Theta(N_I^{+})}{\kappa^{-1}\left(\Theta(N^{!+}_{I^{!}})\right)} S_{I,J}
\end{multline}
\begin{Lemma}\label{lem2}
If $\zeta_i$ or $u_i$ is shifted by $q$, then $\textbf{Stab}_{I,J}$ is scaled by a factor of $\hbar^{1-n}$ and $(q/\hbar)^{n-1}$, respectively.
\end{Lemma}
\begin{proof}
This follows from a direct computation with (\ref{stabs}).
\end{proof}

\begin{Proposition}\label{lem3}
For each $I\prec J$, the operator $D_r(\uu;q,q/\hbar)$ acts diagonally on
$$
 \textbf{Stab}_{I,J} \widetilde{V}_{J}(\aa,\zz)
$$
with eigenvalue $e_r(\zzeta^{-1})$.
\end{Proposition}
\begin{proof}
This follows from Lemma \ref{dopu} and Lemma \ref{lem2}.
\end{proof}

We also need the following lemma.
\begin{Lemma}\label{lem}
$$
\frac{\Theta(N_I^{+}) \Phi((q-\hbar)T^{1/2}_I X)}{\Theta(T^{1/2}_I(X))} = \sqrt{\frac{\det N_I^+}{\det T_I^{1/2}X}} \Phi((q-\hbar)N_I^+)
$$
\end{Lemma}
\begin{proof}
Formula (106) in \cite{AOElliptic} gives 
$$
\frac{\Phi((q-\hbar)T^{1/2}X)}{\Theta(T^{1/2}X)} = \frac{(\det T^{1/2}X)^{-1/2}}{\Phi(TX^{\vee})}
$$
So the left hand side of the Lemma is equal to
$$
\frac{\Theta(N_I^+)}{(\det T^{1/2}_I X)^{1/2} \Phi(T_IX^{\vee})}
$$
We have
\begin{align*}
\frac{\Theta(N_I^+)}{\Phi(T_I X^{\vee})}&= \prod\limits_{w \in \text{char}_{\bT}(N_I^+)} \frac{ w^{1/2} (1-w^{-1}) \varphi(qw) \varphi(q/w)}{ \varphi(1/w) \varphi(\hbar w)} \\
% &=\prod_{w\in \text{char}_{\bT}(N_I^+)} \frac{w^{1/2} \varphi(qw)}{\varphi(\hbar w)} \\
&= (\det N_I^+)^{1/2} \prod_{w \in \text{char}_{\bT}(N_I^+)} \frac{\varphi(qw)}{\varphi(\hbar w)} \\
&= (\det N_I^{+})^{1/2} \Phi((q-\hbar)N_I^{+})
\end{align*}
from which the result follows.
\end{proof}

\subsection{Proof of Theorem \ref{mainthm}}
If we multiply both sides of Theorem \ref{mainthm} by 
$$
\frac{\Theta(N_{I^{!}}^{!+})}{\Theta(T_{I^{!}}^{1/2}X^{!})} \sqrt{\frac{\det T_{I^{!}}^{1/2} X^{!}}{\det N_{I^{!}}^{!+}}}
$$
and use Lemma \ref{lem}, we see that Theorem \ref{mainthm} is equivalent to
\begin{equation}\label{mainthm3}
\alpha_{I^{!}}(\uu^{!},\zzeta^{!},\hbar^{!}) \Phi((q-\hbar^{!})N_{I^{!}}^{!+}) V_{I^{!}}^{!}(\aa^{!},\zz^{!})=\kappa\left( \sum_{J \in X^{\bT}} A_{I,J} \widetilde{V}_J(\aa,\zz) \right)
\end{equation}
where 
\begin{equation}\label{Astab}
A_{I,J}=\sqrt{\frac{\det T_I^{1/2} X}{\det N_I^+}} \kappa^{-1}\left(\alpha_I^{!}(\uu^{!},\zzeta^{!},\hbar^{!})\right) \text{Stab}_{I,J} \alpha_J(\uu,\zzeta,\hbar)^{-1} \Theta(T^{1/2}_J X)^{-1}
\end{equation}
We proceed with a few lemmas before proving Theorem \ref{mainthm}. For a function $f(\aa)$ of the equivariant parameters, we denote by
$$
\lim_{\aa \to 0} f(\aa):=f(0_{\mathfrak{C}})
$$
where $f(0_{\mathfrak{C}})$ is as in (\ref{0c}), assuming this limit exists. Since $a_i=u_i/u_{i+1}$ and by our choice of chamber $\mathfrak{C}$, this is equivalent to sending each $a_i$ to zero.

\begin{Lemma}\label{lem4}
If $I\prec J$ and $I \neq J$, then
$$
\lim_{\aa \to 0} \frac{1}{\kappa^{-1}\left(\alpha_I^{!}(\uu^{!},\zzeta^{!},\hbar^{!})\right)} A_{I,J} \widetilde{V}_{J}(\aa,\zz) =0
$$
\end{Lemma}
\begin{proof}
Using (\ref{stabs}), we rewrite the term inside the limit as
\begin{multline*}
\frac{1}{\kappa^{-1}\left(\alpha_I^{!}(\uu^{!},\zzeta^{!},\hbar^{!})\right)} \sqrt{\frac{\det T_I^{1/2} X}{\det N_I^+}} \frac{e(\xx_I,\zz)}{e(\xx_J,\zz)} \frac{\alpha_I(\uu,\zzeta,\hbar)}{ \alpha_J(\uu,\zzeta,\hbar)} \frac{\Theta(N_I^+) }{\Theta(T^{1/2}_J X)} S_{I,J}  \widetilde{V}_{J}(\aa,\zz)  \\
= \sqrt{\frac{\det T_I^{1/2} X}{\det N_I^+}} \frac{e(\xx_I,\zz)}{e(\xx_J,\zz)} \frac{\Theta(N_I^+) \Phi((q-\hbar)T_J^{1/2}X) }{\Theta(T^{1/2}_J X)} S_{I,J}  V_{J}(\aa,\zz) 
\end{multline*}
Using the identity
$$
\frac{\Phi((q-\hbar)T^{1/2}_J X)}{\Theta(T^{1/2}_J X)} = \frac{(\det T^{1/2}_J X)^{-1/2}}{\Phi(T^{\vee}_J X^)}
$$
we further rewrite this as
$$
\sqrt{\frac{\det T_I^{1/2} X}{\det N_I^+}} \frac{e(\xx_I,\zz)}{e(\xx_J,\zz)} \frac{\Theta(N_I^+) (\det T^{1/2}_J X)^{-1/2}}{\Phi(T^{\vee}_J X)} S_{I,J}  V_{J}(\aa,\zz) 
$$
Since $|q|<1$, we can calculate the limit as $\aa \to 0$ by substituting $a_i=q^{\lambda_i} a_i$ and sending each $\lambda_i \to \infty$. Writing this substitution as $\aa\to \aa q^{\lambda}$, the $q$-periodicity properties of the terms give the following.
\begin{itemize}
\item The term $S_{I,J}$ is $q$-periodic with respect to shifts of $\aa$, by Lemma \ref{S}.
    \item As $\lambda \to \infty$, the vertex function approaches a limit
    \begin{equation}\label{vlim}
    V_J(\aa q^{\lambda},\zz) \to V_J(0_{\mathfrak{C}},\zz)
    \end{equation}
    by Proposition \ref{limver}.
    \item The substitution $\aa\to \aa q^{\lambda}$ transforms the term $\Phi(T^{\vee}_J X)^{-1}$ to 
     \begin{multline*} 
     \Phi(T^{\vee}_J X)^{-1}|_{\aa \to \aa q^{\lambda}} =  \Phi(T^{\vee}_J X)^{-1} \\
      \prod_{\substack{1\leq j<k\leq n \\ J_k<J_j}} \frac{\left(\hbar \frac{u_{J_k}}{u_{J_j}}\right)_{\lambda_{J_k}+\ldots + \lambda_{J_j-1}}}{\left(q \frac{u_{J_k}}{u_{J_j}}\right)_{\lambda_{J_k}+\ldots + \lambda_{J_j-1}}} \left( \frac{u_{J_k}}{u_{J_j}}\right)^{\lambda_{J_k}+\ldots + \lambda_{J_j-1}} q^{\lambda_{J_k}+\ldots + \lambda_{J_j-1}+1 \choose 2}  \\ \prod_{\substack{1\leq j<k\leq n \\ J_k>J_j}} \frac{ \left( \frac{u_{J_j}}{u_{J_k}}\right)_{\lambda_{J_j}+\ldots + \lambda_{J_k-1}}} {\left( \frac{q}{\hbar}\frac{u_{J_j}}{u_{J_k}}\right)_{\lambda_{J_j}+\ldots + \lambda_{J_k-1}}} \left(\hbar^{-1} \frac{u_{J_j}}{u_{J_k}} \right)^{\lambda_{J_j}+\ldots + \lambda_{J_k-1}} q^{\lambda_{J_j}+\ldots + \lambda_{J_k-1} +1 \choose 2}
     \end{multline*}  
     The term $\Theta(N_I^+)$ is transformed as
     \begin{multline*}
         \Theta(N_I^+)|_{\aa \to \aa q^{\lambda}} \\ = \Theta(N_I^{+}) \prod_{\substack{1\leq j < k \leq n \\ I_k < I_j}} \frac{(-1)^{\lambda_{I_k}+\ldots+ \lambda_{I_j-1}}}{\sqrt{q}^{(\lambda_{I_k}+\ldots+ \lambda_{I_j-1})^2}} \left( \frac{u_{I_k}}{u_{I_j}}\right)^{\lambda_{I_k}+\ldots+ \lambda_{I_j-1}} \\ 
         \prod_{\substack{1\leq j < k \leq n \\ I_k>I_j}}  \frac{(-1)^{\lambda_{I_j}+\ldots+ \lambda_{I_k-1}}}{\sqrt{q}^{(\lambda_{I_j}+\ldots+ \lambda_{I_k-1})^2}} \left( \hbar^{-1} \frac{u_{I_j}}{u_{I_k}}\right)^{\lambda_{I_j}+\ldots+ \lambda_{I_k-1}} 
     \end{multline*}
     Combining these two expressions gives
     \begin{multline}\label{thetaphi}
      \left(  \Phi(T^{\vee}_J X)^{-1} \Theta(N_I^+) \right)|_{\aa \to \aa q^{\lambda}} =\Phi(T^{\vee}_J X)^{-1}  \Theta(N_I^+)  \\ \prod_{\substack{1\leq j<k\leq n \\ J_k<J_j}} \frac{\left(\hbar \frac{u_{J_k}}{u_{J_j}}\right)_{\lambda_{J_k}+\ldots + \lambda_{J_j-1}}}{\left(q \frac{u_{J_k}}{u_{J_j}}\right)_{\lambda_{J_k}+\ldots + \lambda_{J_j-1}}} \left( \sqrt{q} \frac{u_{J_k}}{u_{J_j}}\right)^{\lambda_{J_k}+\ldots + \lambda_{J_j-1}} \\ \prod_{\substack{1\leq j<k\leq n \\ J_k>J_j}} \frac{ \left( \frac{u_{J_j}}{u_{J_k}}\right)_{\lambda_{J_j}+\ldots + \lambda_{J_k-1}}} {\left( \frac{q}{\hbar}\frac{u_{J_j}}{u_{J_k}}\right)_{\lambda_{J_j}+\ldots + \lambda_{J_k-1}}} \left(\sqrt{q}\hbar^{-1} \frac{u_{J_j}}{u_{J_k}} \right)^{\lambda_{J_j}+\ldots + \lambda_{J_k-1}} \\
      \prod_{\substack{1\leq j < k \leq n \\ I_k < I_j}}  \left(- \frac{u_{I_k}}{u_{I_j}}\right)^{\lambda_{I_k}+\ldots+ \lambda_{I_j-1}}\prod_{\substack{1\leq j < k \leq n \\ I_k>I_j}}  \left( -\hbar^{-1} \frac{u_{I_j}}{u_{I_k}}\right)^{\lambda_{I_j}+\ldots+ \lambda_{I_k-1}} 
     \end{multline}
    \item Substituting $\aa \to \aa q^{\lambda}$ transforms $e(\xx_I,\zz)/e(\xx_J,\zz)$ into
    $$
    \frac{e(\xx_I,\zz)}{e(\xx_J,\zz)}\bigg|_{\aa \to \aa q^{\lambda}} =  \frac{e(\xx_I,\zz)}{e(\xx_J,\zz)}  \prod_{i=1}^{n-1} \prod_{j=1}^{i} z_i^{\lambda_{I_j}+\ldots + \lambda_{n-1} -(\lambda_{J_j}+\ldots + \lambda_{n-1})} 
    $$
    and this combines with the $q$-contribution from 
    $$
    \sqrt{\frac{\det T^{1/2}_I X}{\det N_I^+ \det T^{1/2}_J X}}
    $$
    to give
    \begin{multline}\label{edet}
   \left(\frac{e(\xx_I,\zz)}{e(\xx_J,\zz)} \sqrt{\frac{\det T^{1/2}_I X}{\det N_I^+ \det T^{1/2}_J X}} \right)\bigg|_{\aa \to \aa q^{\lambda}} = \frac{e(\xx_I,\zz)}{e(\xx_J,\zz)} \sqrt{\frac{\det T^{1/2}_I X}{\det N_I^+ \det T^{1/2}_J X}} \\ \prod_{i=1}^{n-1} \prod_{j=1}^{i} \left(\frac{z_i}{\sqrt{q}}\right)^{\lambda_{I_j}+\ldots + \lambda_{n-1} -(\lambda_{J_j}+\ldots + \lambda_{n-1})} \prod_{i=1}^{n-1} \sqrt{q}^{i\left(\lambda_{I_{i+1}}+\ldots + \lambda_{n-1} -(\lambda_{J_{i+1}}+\ldots \lambda_{n-1})\right)}
    \end{multline}
\end{itemize}
Now we take the limit as $\lambda \to \infty$. The $q$-Pochammer terms in (\ref{thetaphi}) converge to
     \begin{equation}\label{qphlim}
         \prod_{\substack{1\leq j<k\leq n \\ J_k<J_j}} \frac{\varphi\left( \hbar \frac{u_{J_k}}{u_{J_j}} \right)}{\varphi\left( q \frac{u_{J_k}}{u_{J_j}} \right)} \prod_{\substack{1\leq j<k\leq n \\ J_k>J_j}} \frac{\varphi\left( \frac{u_{J_j}}{u_{J_k}} \right)}{\varphi\left( \frac{q}{\hbar} \frac{u_{J_j}}{u_{J_k}} \right)}    = \frac{1}{\Phi((q-\hbar) N_J^+)}
    \end{equation}
    Combining the remaining terms in (\ref{thetaphi}) and (\ref{edet}) gives
    \begin{multline}\label{eq}
        \prod_{i=1}^{n-1} \prod_{j=1}^{i} \left(\frac{z_i}{q}\right)^{\lambda_{I_j}+\ldots + \lambda_{n-1} -(\lambda_{J_j}+\ldots + \lambda_{n-1})} \\
\prod_{\substack{1\leq j < k \leq n \\ I_k < I_j}}  \left(- \frac{u_{I_k}}{u_{I_j}}\right)^{\lambda_{I_k}+\ldots+ \lambda_{I_j-1}}\prod_{\substack{1\leq j < k \leq n \\ I_k>I_j}}   \left( -\hbar^{-1} \frac{u_{I_j}}{u_{I_k}}\right)^{\lambda_{I_j}+\ldots+ \lambda_{I_k-1}}  \\
 \prod_{\substack{1\leq j<k\leq n \\ J_k<J_j}} \left(  \frac{u_{J_k}}{u_{J_j}}\right)^{\lambda_{J_k}+\ldots + \lambda_{J_j-1}}    \prod_{\substack{1\leq j<k\leq n \\ J_k>J_j}} \left(\hbar^{-1} \frac{u_{J_j}}{u_{J_k}} \right)^{\lambda_{J_j}+\ldots + \lambda_{J_k-1}}
    \end{multline}
By the definition of $I\prec J$ in Definition \ref{bruhat}, the power on $z_i$ in (\ref{eq}) is positive. So putting together (\ref{vlim}), (\ref{qphlim}), and (\ref{eq}), we see that in the neighborhood given by $|z_i|<|q|$ and $|a_i|< \max(1,|\hbar|)$
 $$
\lim_{\lambda \to \infty} \left(\frac{1}{\kappa^{-1}\left(\alpha_I^{!}(\uu^{!},\zzeta^{!},\hbar^{!})\right)} A_{I,J} \widetilde{V}_{J}(\aa,\zz) \right)\bigg|_{\aa=\aa q^{\lambda}} =0
 $$
It suffices to consider $|z_i|<|q|$ since the expression is meromorphic in $z_i$. This completes the proof.
\end{proof}

\begin{Lemma}\label{lt}
$$
\lim_{\aa \to 0} \frac{1}{\kappa^{-1}\left(\alpha_I^{!}(\uu^{!},\zzeta^{!},\hbar^{!})\right)} A_{I,I} \widetilde{V}_I(\aa,\zz) = V_I(0_{\mathfrak{C}},\zz)
$$
\end{Lemma}
\begin{proof}
Observe that
\begin{multline*}
\frac{1}{\kappa^{-1}\left(\alpha_I^{!}(\uu^{!},\zzeta^{!},\hbar^{!})\right)} A_{I,I} \widetilde{V}_I(\aa,\zz) \\ =\frac{1}{\alpha_J(\uu,\zzeta,\hbar)} \sqrt{\frac{\det T^{1/2}_I X}{\det N_I^+}} \frac{\Theta(N_I^+)}{\Theta(T^{1/2}_I X)} \widetilde{V}_I(\aa,\zz) \\
= \sqrt{\frac{\det T^{1/2}_I X}{\det N_I^+}} \frac{\Theta(N_I^+) \Phi((q-\hbar)T^{1/2}_I X)}{\Theta(T^{1/2}_I X)} V_I(\aa,\zz)
\end{multline*}
By Lemma \ref{lem}, this is equal to
$$
\Phi((q-\hbar)N_I^+)  V_I(\aa,\zz)
$$
Now we examine the limit as the attracting weights are sent to $0$. Since
$$
\lim_{w\to 0} \frac{\varphi(q w)}{\varphi(\hbar w)} = 1
$$
it follows that 
$$
\lim_{\aa \to 0} \Phi((q-\hbar)N_I^{+}) =1
$$
So the only contribution for the diagonal term is $V_I(0_{\mathfrak{C}},\zz)$ 
% where $\mathfrak{C}$. After applying $\kappa$, we see that
% $$
% \kappa(c_0(\zz)) = \Phi((q-\hbar^{!}) N_{I^{!}}^{!+})
% $$ by Proposition \ref{limver}.
\end{proof}
\begin{proof}[Proof of Theorem \ref{mainthm}]
From (\ref{Astab}), we have
\begin{multline*}
\frac{1}{\kappa^{-1}\left(\alpha_I^{!}(\uu^{!},\zzeta^{!},\hbar^{!})\right)} \sum_{J \in X^{\bT}} A_{I,J}  \widetilde{V}_{J}(\aa,\zz) \\ =  \sum_{J \in X^{\bT}} \sqrt{\frac{\det T_I^{1/2} X}{\det N_I^+}}  \text{Stab}_{I,J} \frac{\Phi((q-\hbar) T^{1/2}_J X)}{ \Theta(T_J^{1/2} X)} V_J(\aa,\zz)
\end{multline*}
By the pole subtraction property of elliptic stable envelopes, see Theorem 5 in \cite{AOElliptic}, the expression on the right hand side of the above equation is holomorphic in a neighborhood of $a_i=0$. Expanding as a power series in $\aa$, we get
$$
\frac{1}{\kappa^{-1}\left(\alpha_I^{!}(\uu^{!},\zzeta^{!},\hbar^{!})\right)} \sum_{J \in X^{\bT}} A_{I,J}  \widetilde{V}_{J}(\aa,\zz) = \sum_{\substack{\dd \\ d_i \geq 0}} c_\dd(\zz) \aa^\dd 
$$
Clearly,
$$
c_0(\zz) = \lim_{\aa\to 0} \frac{1}{\kappa^{-1}\left(\alpha_I^{!}(\uu^{!},\zzeta^{!},\hbar^{!})\right)} \sum_{J \in X^{\bT}} A_{I,J}  \widetilde{V}_{J}(\aa,\zz)
$$
By Lemmas \ref{lem4} and \ref{lt}, this limit is
$$
c_0(\zz)=V_{I}(0_{\mathfrak{C}},\zz)
$$
So we have
\begin{equation}\label{expr2}
\frac{1}{V_{I}(0_{\mathfrak{C}},\zz)}\sum_{J \in X^{\bT}} A_{I,J}  \widetilde{V}_{J}(\aa,\zz) = \kappa^{-1}\left(\alpha_I^{!}(\uu^{!},\zzeta^{!},\hbar^{!})\right)  \sum_{\substack{\dd \\ d_i \geq 0}} c'_\dd(\zz) \aa^\dd 
\end{equation}
where $c'_0(\zz)=1$. By Proposition \ref{lem3}, (\ref{expr2}) is an eigenvector of $D_r(\uu;q,q/\hbar)$ with eigenvalue $e_r(\zzeta^{-1})$ for $r \in \{1,\ldots,n\}$. Furthermore, formula (2.23) in \cite{NSmac} gives a recursive formula for the coefficients of a solution to the joint eigenvalue problem for the operators $D_r(\uu;q,q/\hbar)$. Since $c'_0(\zz)=1$, the form of the recursion evidently implies that $c'_d(\zz)\in \mathbb{C}_{q,\hbar}(\zz)$ for all $d$, where $\mathbb{C}_{q,\hbar}=\mathbb{C}(q,\hbar)$. Equivalently,
$$
\kappa^{-1}\left(\alpha_I^{!}(\uu^{!},\zzeta^{!},\hbar^{!})\right)  \sum_{\substack{\dd \\ d_i \geq 0}} c'_\dd(\zz) \aa^\dd  \in \kappa^{-1}\left(\alpha_I^{!}(\uu^{!},\zzeta^{!},\hbar^{!})\right)  \mathbb{C}_{q,\hbar}(\zz)[[\aa]]
$$

Proposition \ref{dopzeta} implies that
\begin{equation}\label{expr}
\kappa^{-1}\left( \alpha_{I^{!}}(\uu^{!},\zzeta^{!},\hbar^{!}) V^{!}_{I^{!}}(\uu^{!},\zzeta^{!}) \right)
\end{equation}
is also a solution to the joint eigenvector problem for $D_r(\uu;q,q/\hbar)$ for $r \in \{1,\ldots,n\}$. It is obvious from Theorem \ref{ver} that
$$
\kappa^{-1}\left( \alpha_{I^{!}}(\uu^{!},\zzeta^{!},\hbar^{!}) V^{!}_{I^{!}}(\uu^{!},\zzeta^{!}) \right) \in \kappa^{-1}\left(\alpha_I^{!}(\uu^{!},\zzeta^{!},\hbar^{!})\right)  \mathbb{C}_{q,\hbar}(\zz)[[\aa]]
$$

By Theorem 2.1 in \cite{NSmac}, solutions of the joint eigenvector problem for $D_r(\uu;q,q/\hbar)$ for $r \in \{1,\ldots,n\}$ inside the ring $\kappa^{-1}\left(\alpha_I^{!}(\uu^{!},\zzeta^{!},\hbar^{!})\right) \mathbb{C}_{q,\hbar}(\zz)[[\aa]]$ are uniquely determined by their leading coefficient in $\aa$. 

We have already shown that $c'_0(\zz)=1$, and it is obvious from Theorem \ref{ver} that (\ref{expr}) also has leading coefficient 1. Hence we must have
$$
\frac{1}{V_{I}(0_{\mathfrak{C}},\zz)}\sum_{J \in X^{\bT}} A_{I,J}  \widetilde{V}_{J}(\aa,\zz) = \kappa^{-1}\left( \alpha_{I^{!}}(\uu^{!},\zzeta^{!},\hbar^{!}) V^{!}_{I^{!}}(\uu^{!},\zzeta^{!}) \right)
$$
Multiypling by $V_{I}(0_{\mathfrak{C}},\zz)$, applying $\kappa$, and using Proposition \ref{limver} gives (\ref{mainthm3}), which finishes the proof.

\end{proof}

\section{3d mirror symmetry of elliptic stable envelopes}
We can use Theorem \ref{mainthm} to recover the main result of \cite{MirSym2}. Let $P$ be the permutation matrix that identifies the fixed points of $X$ and $X^{!}$. Explicitly,
$$
P_{I,J}=\begin{cases}
1 & \text{if } I=J^{-1} \\
0 & \text{otherwise}
\end{cases}
$$
Let $\textbf{Stab}^{!}$ be the elliptic stable envelope for $X^{!}$ for the chamber $-\mathfrak{C}^{!}$ and polarization $T^{1/2}X^{!}$. It is equivalent to $\textbf{Stab}$, but with the equivariant and K\"ahler parameters trivially swapped
$$
\zeta \mapsto \zeta^{!} \qquad \text{and} \qquad u \mapsto u^{!}
$$
Then Theorem \ref{mainthm} gives
$$
\widetilde{V}_I(\aa,\zz)= \kappa^{-1}\left(\sum_{J^{!}} \textbf{Stab}^{!}_{I^{!},J^{!}} \widetilde{V}_{J^{!}}^{!}(\aa^{!},\zz^{!}) \right)
$$
Writing $\widetilde{V}(\aa,\zz)$ for the vector of vertex functions of $X$, and similarly for $X^{!}$, we see that
\begin{equation}\label{matform}
\widetilde{V}(\aa,\zz)=M(\aa,\zz) \widetilde{V}(\aa,\zz)
\end{equation}
where 
$$
M(\aa,\zz)=P^{-1}\kappa^{-1}\left( \textbf{Stab}^{!}\right) P \textbf{Stab}
$$
\begin{Corollary}\label{stabms}
If we write $M_{I,J}$ for the entries in the matrix $M(\aa,\zz)$, then we have
$$
M_{I,J}=\delta_{I,J}
$$
In other words
$$
\textbf{Stab}^{-1}= P^{-1} \kappa^{-1}\left(\textbf{Stab}^{!}\right) P
$$
\end{Corollary}
\begin{proof}
From (\ref{matform}), we have an equation
$$
\widetilde{V}_{I}(\aa,\zz)=\sum_{J\in X^{\bT}} M_{I,J} \widetilde{V}_J(\aa,\zz)
$$
Now, since the entries in the elliptic stable envelopes are meromorphic functions of all parameters, we deduce that $M_{I,J}$ are meromorphic functions of $\aa$ and $\zz$. Furthermore, the known $q$-periodicity of the elliptic stable envelope implies that $M_{I,J}$ is $q$-periodic with respect to shifts of $\aa$ and $\zz$. 

From Definition \ref{boldstab}, we see that $M_{I,I}=1$ for all $I\in X^{\bT}$. Along with triangularity of stable envelopes, this gives us
$$
\sum_{\substack{J \succ I\\J\neq I}} M_{I,J} \widetilde{V}_J(\aa,\zz)=0
$$
which is equivalent to 
$$
\sum_{\substack{J \succ I\\J\neq I}} M_{I,J} \alpha_{J}(\zzeta,\uu,\hbar) \Phi((q-\hbar)T_J^{1/2}X)V_J(\aa,\zz)=0
$$
Let $K$ be the minimal index of the sum with respect to the order $\prec$, and divide both sides by $\alpha_{K}(\zzeta,\uu,\hbar)$ to get
$$
\sum_{\substack{J \succ I\\J\neq I}} M_{I,J} \frac{\alpha_{J}(\zzeta,\uu,\hbar)}{\alpha_{K}(\zzeta,\uu,\hbar)} \Phi((q-\hbar)T_J^{1/2}X)V_J(\aa,\zz)=0
$$
Substituting $z_i=z_i q^{\lambda_i}$ for $\lambda_i\in \mathbb{N}$ and using the known $q$-periodicity gives
\begin{equation}\label{meq}
\sum_{\substack{J \succ I\\J\neq I}} M_{I,J} \frac{\alpha_{J}(\zzeta,\uu,\hbar)}{\alpha_{K}(\zzeta,\uu,\hbar)} \prod_{i=1}^{n-1} \left(\frac{\mathcal{L}_i|_{K}}{\mathcal{L}_i|_{J}}\right)^{\lambda_i} \Phi((q-\hbar)T_J^{1/2}X)V_J(\aa ,\zz q^{\lambda})=0
\end{equation}
Observe that if $K\prec J$, then 
$$
\frac{\mathcal{L}_i|_{K}}{\mathcal{L}_i|_{J}}
$$
is either $1$ or an attracting weight. Furthermore, there is some $i$ for which it is attracting. Thus, assuming that $|a_i|<1$, we see that
$$
\lim_{\lambda\to\infty} \prod_{i=1}^{n-1} \left(\frac{\mathcal{L}_i|_{K}}{\mathcal{L}_i|_{J}}\right)^{\lambda_i} =0
$$
If we take the limit $\lambda\to\infty$ in (\ref{meq}), only one term survives. Furthermore, since $|q|<1$, it is clear from Theorem \ref{ver} that
$$
\lim_{\lambda \to \infty} V_{J}(\aa,\zz q^{\lambda}) = 1
$$
So taking the limit as $\lambda \to \infty$ in (\ref{meq}) gives
$$
M_{I,K} \Phi((q-\hbar)T_K^{1/2}X)=0
$$
Since the zeros and poles of the latter term are both isolated, this implies that 
$$
M_{I,K}=0
$$
Repeating this argument inductively implies that $M_{I,J}=0$ whenever $I\neq J$, which yields the result.
\end{proof}

This Corollary is precisely what is needed to prove the existence of the so-called duality interface, see \cite{KS2}. The elliptic stable envelopes of $X$ and $X^{!}$ glue to give an elliptic cohomology class on the product $X\times X^{!}$ which restricts to the elliptic stable envelopes of each. In the language of \cite{KS2}, Theorem \ref{mainthm} says that the correspondence given by the duality interface maps the vertex functions of $X$ to the vertex functions of $X^{!}$.

The proof of Corollary \ref{stabms} in \cite{MirSym2} relies on very special combinatorial properties of the elliptic weight functions which do not hold in general. In contrast, we expect that the above proof can be generalized to show that the 3d mirror symmetry of vertex functions implies the 3d mirror symmetry of elliptic stable envelopes.

\printbibliography

\newpage

\noindent
Hunter Dinkins\\
Department of Mathematics,\\
University of North Carolina at Chapel Hill,\\
Chapel Hill, NC 27599-3250, USA\\
hdinkins@live.unc.edu 

\end{document}